\documentclass[11pt]{article}
\usepackage{amsmath, amscd, amssymb, latexsym, epsfig, color, amsthm}
\usepackage{tikz, enumerate}
\usetikzlibrary{calc}
\usepackage{pstricks,graphicx,subfig}
\numberwithin{equation}{section}

\newtheorem{theorem}{Theorem}[section]
\newtheorem{proposition}[theorem]{Proposition}

\newtheorem{corollary}[theorem]{Corollary}
\newtheorem{conjecture}[theorem]{Conjecture}
\newtheorem{lemma}[theorem]{Lemma}

\theoremstyle{definition}

\newtheorem{example}[theorem]{Example}

\newtheorem{remark}[theorem]{Remark}

\DeclareMathOperator\lk{\mathrm{lk}}

\newcommand{\cupdot}{\mathbin{\mathaccent\cdot\cup}}

\newcommand{\field}{{\bf k}}

\newcommand{\Sp}{\mathbb{S}}

\title{The flag upper bound theorem for 3- and 5-manifolds}
\author{Hailun Zheng\\
	\small Department of Mathematics \\[-0.8ex]
	\small University of Washington\\[-0.8ex]
	\small Seattle, WA 98195-4350, USA\\[-0.8ex]
	\small \texttt{hailunz@math.washington.edu}
}
\begin{document}
	\maketitle
	\begin{abstract}
		We prove that among all flag 3-manifolds on $n$ vertices, the join of two circles with $\left\lceil{\frac{n}{2}}\right \rceil$ and $\left\lfloor{\frac{n}{2}}\right \rfloor$ vertices respectively is the unique maximizer of the face numbers. This solves the first case of a conjecture due to Lutz and Nevo. Further, we establish a sharp upper bound on the number of edges of flag 5-manifolds and characterize the cases of equality. We also show that the inequality part of the flag upper bound conjecture continues to hold for all flag 3-dimensional Eulerian complexes and find all maximizers of the face numbers in this class.
	\end{abstract}
	\section{Introduction}
	One of the classical problems in geometric combinatorics deals with the following question: for a given class of simplicial complexes, find tight upper bounds on the number of $i$-dimensional faces as a function of the number of vertices and the dimension. Since Motzkin \cite{Mo} proposed the upper bound conjecture (UBC, for short) for polytopes in 1957, this problem has been solved for various families of complexes. In particular, McMullen \cite{M} and Stanley \cite{St} proved that neighborly polytopes simultaneously maximize all face numbers in the class of polytopes and simplicial spheres. However, it turns out that, apart from cyclic polytopes, many other classes of neighborly spheres or even neighborly polytopes exist, see \cite{Sh} and \cite{P} for examples and constructions of neighborly polytopes.
	
	A simplicial complex $\Delta$ is \emph{flag} if all of its minimal non-faces have cardinality two, or equivalently, $\Delta$ is the clique complex of its graph. Flag complexes form a beautiful and important class of simplicial complexes. For example, barycentric subdivisions of simplicial complexes, order complexes of posets, and Coxeter complexes are flag complexes. Despite a lot of effort that went into studying the face numbers of flag spheres, in particular in relation with the Charney-Davis conjecture \cite{CD}, and its generalization given by Gal's conjecture \cite{G}, a flag upper bound theorem for spheres is still unknown. The upper bounds on face numbers for general simplicial $(d-1)$-spheres are far from being sharp for flag $(d-1)$-spheres, since the graph of any flag $(d-1)$-dimensional complex is $K_{d+1}$-free. In \cite{G}, Gal confirmed that the real rootedness conjecture \cite{DDJO} does hold for flag homology spheres of dimension less than five, and thus the upper bounds on the face numbers of these flag spheres were established. However, starting from dimension five, there are only conjectural upper bounds. For $m\geq 1$, we let $J_m(n)$ be the $(2m-1)$-sphere on $n$ vertices obtained as the join of $m$ copies of the circle, each one a cycle with either $\lfloor \frac{n}{m} \rfloor$ or $\lceil \frac{n}{m}\rceil$ vertices. The following conjecture is due to Nevo and Petersen \cite[Conjecture 6.3]{NP}. 
	\begin{conjecture}\label{conj: Nevo-Petersen}
		If $\Delta$ is a flag homology sphere, then $\gamma(\Delta)$ satisfies the Frankl-F{\"u}redi-Kalai inequalities. In particular, if $\Delta$ is of dimension $d=2m-1$, then $f_i(\Delta)\leq f_i(J_m(n))$ for all $1\leq i\leq 2m-1$.
	\end{conjecture}
	\noindent Here we denote by $f_i(\Delta)$ the number of $i$-dimensional faces of $\Delta$; the entries of the vector $\gamma(\Delta)$ are certain linear combinations of the $f$-numbers of $\Delta$. For Frankl-F{\"u}redi-Kalai inequalities, see \cite{FFK}.
	
	As for the case of equality, Lutz and Nevo \cite[Conjecture 6.3]{LN} posited that, as opposed to the case of all simplicial spheres, for a fixed dimension $2m-1$ and the number of vertices $n$, there is \emph{only} one maximizer of the face numbers.
	\begin{conjecture}\label{conj: Lutz-Nevo}
		Let $m\geq 2$ and let $\Delta$ be a flag simplicial $(2m-1)$-sphere on $n$ vertices. Then $f_i(\Delta)=f_i(J_m(n))$ for some $1\leq i\leq m$ if and only if $\Delta=J_m(n)$.
	\end{conjecture}
	Recently, Adamaszek and Hladk{\'y} \cite{AH} proved that Conjecture \ref{conj: Lutz-Nevo} holds asymptotically for flag homology manifolds. Several celebrated theorems from extremal graph theory served as tools for their work. As a result, the proof simutaneously gives upper bounds on $f$-numbers, $h$-numbers, $g$-numbers and $\gamma$-numbers, but it only applies to flag homology manifolds with an extremely large number of vertices.
	
	Our first main result is that Conjecture \ref{conj: Nevo-Petersen} and \ref{conj: Lutz-Nevo} hold for \emph{all} flag 3-manifolds. In particular, we show that the balanced join of two circles is the unique maximizer of face numbers in the class of flag 3-manifolds. We also establish an analogous result on the number of edges of flag 5-manifolds. The proof only relies on simple properties of flag complexes and Eulerian complexes. 
	
	In 1964, Klee \cite{K} proved that Motzkin's UBC for polytopes holds for a much larger class of Eulerian complexes as long as they have sufficiently many vertices, and conjectured that the UBC holds for \emph{all} Eulerian complexes. Our second main result deals with flag Eulerian complexes, and asserts that Conjecture \ref{conj: Nevo-Petersen} continues to hold for all flag 3-dimensional Eulerian complexes. This provides supporting evidence to a question of Adamaszek and Hladk{\' y} \cite[Problem 17(i)]{AH} in the case of dimension 3, where they proposed that Conjecture \ref{conj: Nevo-Petersen} holds for all odd-dimensional flag weak pseudomanifolds with sufficiently many vertices. We also give constructions of the maximizers of face numbers in this class and show that they are the only maximizers. Our proof is based on an application of the inclusion-exclusion principle and double counting.
	
	The paper is organized as follows. In Section 2, we discuss basic facts on simplicial complexes and flag complexes. In Section 3, we provide the proof of our first main result asserting that given a number of vertices $n$, the maximum face numbers of a flag 3-manifold are achieved only when this manifold is the join of two circles of length as close as possible to $\frac{n}{2}$. In Section 4, we apply an analogous argument to the class of flag 5-manifolds. In Section 5, we show that the same upper bounds continue to hold for the class of flag 3-dimensional Eulerian complexes, and discuss the maximizers of the face numbers in this class. Finally, we close in Section 6 with some concluding remarks. 
	
	\section{Preliminaries}
	A \emph{simplicial complex} $\Delta$ on a vertex set $V=V(\Delta)$ is a collection of subsets
	$\sigma\subseteq V$, called faces, that is closed under inclusion. For $\sigma\in \Delta$, let $\dim\sigma:=|\sigma|-1$ and define the \emph{dimension} of $\Delta$, $\dim \Delta$, as the maximal dimension of its faces. A \emph{facet} in $\Delta$ is a maximal face under inclusion, and we say that $\Delta$ is \emph{pure} if all of its facets have the same dimension. We will denote by $\sqcup$ the disjoint union of simplicial complexes, and by $\cupdot$ the disjoint union of sets.
	
	If $\Delta$ is a simplicial complex and $\sigma$ is a face of $\Delta$, the \emph{link} of $\sigma$ in $\Delta$ is $\lk_\Delta \sigma:=\{\tau-\sigma\in \Delta: \sigma\subseteq \tau\in \Delta\}$, and the \emph{deletion} of a vertex set $W$ from $\Delta$ is $\Delta\backslash W:=\{\sigma\in\Delta:\sigma\cap W=\emptyset\}$. The \emph{restriction} of $\Delta$ to a vertex set $W$ is defined as $\Delta[W]:=\{\sigma\in \Delta:\sigma\subseteq W\}$. If $\Delta$ and $\Gamma$ are two simplicial complexes on disjoint vertex sets, then the \emph{join} of $\Delta$ and $\Gamma$, denoted as $\Delta *\Gamma$, is the simplicial complex on vertex set $V(\Delta)\cupdot V(\Gamma)$ whose faces are $\{\sigma\cup\tau:\sigma\in\Delta, \tau\in\Gamma\}$.
	
	A simplicial complex $\Delta$ is a \emph{simplicial manifold} (resp. \emph{simplicial sphere}) if the geometric realization of $\Delta$ is homeomorphic to a manifold (resp. sphere). We denote by $\tilde{H}_*(\Delta;\field)$ the reduced homology of $\Delta$ computed with coefficients in a field $\field$, and by $\beta_i(\Delta;\field):=\dim_{\field}\tilde{H}_i(\Delta;\field)$ the reduced Betti numbers of $\Delta$ with coefficients in $\field$. We say that $\Delta$ is a $(d-1)$-dimensional \emph{$\field$-homology manifold} if $\tilde{H}_*(\lk_\Delta \sigma;\field)\cong \tilde{H}_*(\mathbb{S}^{d-1-|\sigma|};\field)$ for every nonempty face $\sigma\in\Delta$. A $\field$-\emph{homology sphere} is a $\field$-homology manifold that has the $\field$-homology of a sphere. Every simplicial manifold (resp. simplicial sphere) is a homology manifold (resp. homology sphere). Moreover, in dimension two, the class of homology 2-spheres coincides with that of simplicial 2-spheres, and hence in dimension three, the class of homology 3-manifolds coincides with that of simplicial manifolds.
	
	For a $(d-1)$-dimensional complex $\Delta$, we let $\chi(\Delta):=\sum_{i=0}^{d-1}(-1)^i \beta_i(\Delta;\field)$ be the \emph{reduced Euler characteristic} of $\Delta$. A simplicial complex $\Delta$ is called an \emph{Eulerian} complex if $\Delta$ is pure and $\chi(\lk_\Delta \sigma)=(-1)^{\dim \lk_\Delta \sigma}$ for every $\sigma\in\Delta$, including $\sigma=\emptyset$. In particular, it follows from the Poincar$\mathrm{\acute{e}}$ duality theorem that all odd-dimensional orientable homology manifolds are Eulerian. As all simplicial manifolds are orientable over $\mathbb{Z}/2\mathbb{Z}$, all odd-dimensional simplicial manifolds are Eulerian.
	
	A $(d-1)$-dimensional simplicial complex $\Delta$ is called a \emph{weak $(d-1)$-pseudomanifold} if it is pure and every $(d-2)$-face (called \emph{ridge}) of $\Delta$ is contained in exactly two facets. A weak $(d-1)$-pseudomanifold $\Delta$ is called a \emph{normal $(d-1)$-pseudomanifold} if it is connected, and the link of each face of dimension $\leq d-3$ is also connected. Every Eulerian complex is a weak pseudomanifold, and every connected homology manifold is a normal pseudomanifold. In fact, every normal 2-pseudomanifold is also a homology 2-manifold. However, for $d>3$, the class of normal $(d-1)$-pseudomanifolds is much larger than the class of homology $(d-1)$-manifolds. It is well-known that if $\Delta$ is a weak (resp. normal) $(d-1)$-pseudomanifold and $\sigma$ is a face of $\Delta$ of dimension at most $d-2$, then the link of $\sigma$ is also a weak (resp. normal) pseudomanifold. The following lemma gives another property of normal pseudomanifolds, see \cite[Lemma 1.1]{BD}.
	
	\begin{lemma}\label{lemma: normal pseudomanifold}
		Let $\Delta$ be a normal $(d-1)$-pseudomanifold, and let $W$ be a subset of vertices of $\Delta$ such that the induced subcomplex $\Delta[W]$ is a normal $(d-2)$-pseudomanifold. Then the induced subcomplex of $\Delta$ on vertex set $V(\Delta)\backslash W$ has at most two connected components.
	\end{lemma}
	
	For a $(d-1)$-dimensional complex $\Delta$, we let $f_i = f_i(\Delta)$ be the number of $i$-dimensional faces of $\Delta$ for $-1\leq i\leq d-1$. The vector $(f_{-1}, f_0, \cdots, f_{d-1})$ is called the $f$\emph{-vector} of $\Delta$. Since the graph of any simplicial 2-sphere is a maximal planar graph, it follows that the $f$-vector  of a simplicial 2-sphere is uniquely determined by $f_0$. For a 3-dimensional Eulerian complex, the following lemma indicates that its $f$-vector is uniquely determined by $f_0$ and $f_1$.
	\begin{lemma}\label{f number:3-manifold}
		The $f$-vector of a 3-dimensional Eulerian complex satisfies  \[(f_0,f_1,f_2,f_3)=(f_0,f_1,2f_1-2f_0, f_1-f_0).\]
	\end{lemma}
	\begin{proof}
		Let $\Delta$ be a 3-dimensional Eulerian complex. Since $\Delta$ is Eulerian, $\chi(\Delta)+1=f_0-f_1+f_2-f_3=0$. Also since every ridge of an Eulerian complex is contained in exactly two facets, by double counting, we obtain that $2f_2=4f_3$. Hence the result follows. 
	\end{proof}
	
	A simplicial complex $\Delta$ is \emph{flag} if all minimal non-faces of $\Delta$, also called missing faces, have cardinality two; equivalently, $\Delta$ is the clique complex of its graph. The following lemma \cite[Lemma 5.2]{NP} gives a basic property of flag complexes.
	\begin{lemma}\label{flag prop}
		Let $\Delta$ be a flag complex on vertex set $V$. If $W\subseteq V$, then $\Delta[W]$ is also flag. Furthermore, if $\sigma$ is a face in $\Delta$, then $\lk_\Delta \sigma=\Delta[V(\lk_\Delta \sigma)]$. In particular, all links in a flag complex are also flag.
	\end{lemma}
	
	Finally, we recall some terminology from graph theory. A graph $G$ is a path graph if the set of its vertices can be ordered as $x_1,x_2,\cdots,x_n$ in such a way that $\{x_i,x_{i+1}\}$ is an edge for all $1\leq i\leq n-1$ and there are no other edges. Similarly, a cycle graph is a graph obtained from a path graph by adding an edge between the endpoints of the path.
	
	\section{The Proof of the Lutz-Nevo Conjecture for flag 3-manifolds}
	
	The goal of this section is to prove Conjecture \ref{conj: Lutz-Nevo} for 3-dimensional manifolds. We start by setting up some notation and establishing several lemmas that will be used in the proof. Recall that in the Introduction, we defined $J_m(n)$ to be the $(2m-1)$-sphere on $n$ vertices obtained as the join of $m$ circles, each one of length either $\lfloor \frac{n}{m} \rfloor$ or $\lceil \frac{n}{m}\rceil$. In the following, we define $J_m^*(n)$ as the suspension of $J_m(n-2)$, and denote by $\mathcal{C}_d^*$ the $(d-1)$-dimensional octahedral sphere, i.e., the boundary complex of the $d$-dimensional cross-polytope. Equivalently, $C_d^*$ is a $d$-fold join of $\Sp^0$, and thus $\mathcal{C}_{2m}^*=J_m(4m)$ and $\mathcal{C}_{2m+1}^*=J_m^*(4m+2)$. The following lemma \cite[Lemma 5.3]{NP}, originally stated for the class of flag homology spheres, gives a sufficient condition for a flag normal pseudomanifold to be an octahedral sphere. (As the proof is identical to that of \cite[Lemma 5.3]{NP}, we omit it.)
	\begin{lemma}\label{lemma: octahedral sphere}
		Let $\Delta$ be a $(d-1)$-dimensional flag normal pseudomanifold on vertex set $V$ such that for any $v\in V$, $\lk_\Delta v$ is an octahedral sphere. Then $\Delta$ is an octahedral sphere.
	\end{lemma}
	Next we characterize all flag normal $(d-1)$-pseudomanifolds on $\leq 2d+1$ vertices.
	\begin{lemma}\label{lemma: few vertices}
		Let $\Delta$ ba a flag normal $(d-1)$-pseudomanifold, where $d\geq 2$. Then 
		\begin{enumerate}[(a)]
			\item $f_0(\Delta)\geq 2d$. Equality holds if and only if $\Delta=\mathcal{C}_d^*$.
			\item $f_0(\Delta)=2d+1$ if and only if $\Delta=J_1(5)*\mathcal{C}_{d-2}^*$, or equivalently,  $\Delta=J_m(4m+1)$ if $d=2m$, and $\Delta=J_m^*(4m+3)$ if $d=2m+1$.
		\end{enumerate}
	\end{lemma}
	\begin{proof}
		Part (a) is well-known, see the proof of Proposition 2.2 in \cite{At}. For part (b), we use induction on the dimension. 
		If $d=2$ and $f_0(\Delta)=5$, then $\Delta$ is a circle of length 5, that is, $J_1(5)$. Now assume that the claim holds for $d<k$, and consider a flag normal $(k-1)$-pseudomanifold $\Delta$ on $2k+1$ vertices. Since $\lk_\Delta v$ is a flag normal pseudomanifold for any vertex $v\in \Delta$, by part (a), $f_0(\lk_\Delta v)\geq 2k-2$. Moreover, if equality holds, then $\lk_\Delta v$ is an octahedral sphere. However, if $f_0(\lk_\Delta v)=2k-2$ for every vertex $v\in\Delta$, then by Lemma \ref{lemma: octahedral sphere}, $\Delta$ must be the octahedral sphere of dimension $k-1$, contradicting that $f_0(\Delta)=2k+1$. Hence there is at least one vertex $u\in \Delta$ such that $f_0(\lk_\Delta u)\geq 2k-1$. Since $\Delta$ is a normal pseudomanifold, for any facet $\sigma\in\lk_\Delta u$, $\lk_\Delta \sigma$ consists of two vertices. On the other hand, since $\Delta$ is flag, $\lk_\Delta u$ is also flag by Lemma \ref{flag prop}, and hence $\lk_\Delta \sigma$ does not contain any vertex in $V(\lk_\Delta u)$. It follows that $f_0(\lk_\Delta u)\leq (2k+1)-2=2k-1$. Therefore, $f_0(\lk_\Delta u)=2k-1$, and the links of all facets of $\lk_\Delta u$ are the same two vertices. This implies that $\Delta$ is the suspension of $\lk_\Delta u$. By induction, $\lk_\Delta u=J_1(5)*\mathcal{C}_{d-3}^*$, and hence $\Delta=J_1(5)*\mathcal{C}_{d-2}^*$.
	\end{proof}
	
	Now we estimate the number of edges in a flag normal 3-pseudomanifold on $n$ vertices.
	\begin{lemma}\label{lemma: flag normal 3-pseudomanifolds}
		Let $\Delta$ be a flag normal 3-pseudomanifold on $n$ vertices. Then $f_1(\Delta)\leq f_1(J_2(n))+c$, where $c=3-3\min_{v\in \Delta} \chi(\lk_\Delta v)$.
	\end{lemma}
	\begin{proof}
		Let $v$ be a vertex of maximum degree in $V(\Delta)$. We let $a=f_0(\lk_\Delta v)$, $W_1=V(\lk_\Delta v)$ and $W_2=V(\Delta)\backslash V(\lk_\Delta v)$. Since $\Delta$ is a normal 3-pseudomanifold, $\lk_\Delta v$ is a normal 2-pseudomanifold, i.e., a simplicial 2-manifold. Furthermore, since $\Delta$ is flag, by Lemma \ref{flag prop}, $\lk_\Delta v$ is the restriction of $\Delta$ to $W_1$. Thus, by Lemma \ref{lemma: normal pseudomanifold}, the induced subcomplex $\Delta[W_2]$ has at most two connected components. Since $v$ is not connected to any vertices in $W_2\backslash\{v\}$, it follows that $\{v\}$ and $\Delta[W_2\backslash\{v\}]$ are the two connected components in $\Delta[W_2]$.
		
		We now count the edges of $\Delta$. They consist of the edges of $\Delta[W_1]=\lk_\Delta v$, the edges of $\Delta[W_2]$ and the edges between these two sets. In addition, $\sum_{w\in W_2} f_0(\lk_\Delta w)$ counts the edges of $\Delta[W_2]$ twice. Thus,
		\begin{align}\label{formular}
		\begin{split}
	    f_1(\Delta)&=f_1(\Delta[W_1])+\Big(\sum_{w\in W_2} f_0(\lk_\Delta w)\Big)-f_1(\Delta[W_2])\\
		&\stackrel{(*)}{\leq} f_1(\lk_\Delta v)+|W_2|\cdot\max_{w\in W_2} f_0(\lk_\Delta w) -(f_0(\Delta[W_2\backslash\{v\}])-1)\\
		&\stackrel{(**)}{=} \big(3a-6+3(1-\chi(\lk_\Delta v))\big)+(n-a)a-(n-a-2)\\
		&=-a^2+a(n+4)-(n+4)+3-3\chi(\lk_\Delta v)\\
		&\stackrel{(***)}{\leq} \left \lfloor{\frac{n^2}{4}}\right \rfloor+n+3-3\chi(\lk_\Delta v)\\
		&=f_1(J_2(\Delta))+3(1-\chi(\lk_\Delta v)).
		\end{split}
		\end{align}
		Here in (*) we used that $\Delta[W_2\backslash\{v\}]$ is connected and hence has at least $f_0(\Delta[W_2\backslash\{v\}])-1$ edges. Equality (**) follows from the fact that $\lk_\Delta v$ is a 2-manifold with $a$ vertices, and (***) is obtained by optimizing the function $p(a)=-a^2+a(n+4)$. Hence the result follows.
	\end{proof}
    \begin{theorem}\label{theorem: 3-manifold}
    	Let $\Delta$ be a flag 3-manifold on $n$ vertices. Then $f_i(\Delta)\leq f_i(J_2(n))$. If equality holds for some $1\leq i\leq 3$, then $\Delta=J_2(n)$.
    \end{theorem}	
	\begin{proof}
		We use the same notation as in the proof of Lemma \ref{lemma: flag normal 3-pseudomanifolds}. That is, we let $v$ be a vertex of maximum degree in $V(\Delta)$. We let $a=f_0(\lk_\Delta v)$, $W_1=V(\lk_\Delta v)$ and $W_2=V(\Delta)\backslash V(\lk_\Delta v)$. Since $\Delta$ is a flag 3-manifold, $\chi(\lk_\Delta w)=1$ for every $w\in \Delta$. Hence by Lemma \ref{lemma: flag normal 3-pseudomanifolds}, $f_1(\Delta)\leq f_1(J_2(\Delta))$. Furthermore, it follows from steps (*) and (***) in equality (\ref{formular}) that $f_1(\Delta)=f_1(J_2(n))$ holds only if $f_0(\lk_\Delta w)=a=\left\lceil{\frac{n+4}{2}}\right \rceil$ or $\left\lfloor{\frac{n+4}{2}}\right\rfloor$ for all $w\in W_2$, and $\Delta[W_2\backslash\{w\}]$ is a tree.
			
			We claim that if $f_1(\Delta)=f_1(J_2(n))$, then $\Delta=J_2(n)$. This indeed holds if $n=8$ or 9, since by Lemma \ref{lemma: few vertices}, the only flag 3-manifolds on 8 or 9 vertices are $J_2(8)$ and $J_2(9)$. Next we assume that $n\geq 10$, where $|W_2|=n-a\geq \left\lceil{\frac{n}{2}}\right \rceil -2>2$. Hence the tree $\Delta[W_2\backslash\{v\}]$ has at least one edge, and thus there is a vertex $u_1\in W_2$ such that $\deg_{\Delta[W_2]}u_1=1$. Let $u_2$ be the unique vertex in $W_2$ that is connected to $u_1$. Since $f_0(\lk_\Delta u_1)=a$, the vertex $u_1$ must be connected to all vertices in $W_1$ except for one vertex. We let $z_1$ be this vertex and denote the circle $\lk_{\lk_\Delta v}z_1$ by $C_1$. Since $\Delta$ is flag, $\lk_\Delta u_1\supseteq \Delta[W_1\backslash\{z_1\}]=\lk_\Delta v - \{z_1\}* C_1$, and hence \[\lk_\Delta u_1=(\lk_\Delta v-\{z_1\}*C_1) \cup (\{u_2\}*C_1).\]
			
			If $\{z_1\}\in \lk_\Delta u_2$, then $\lk_\Delta u_2\supseteq C_1*\{u_1,z_1\}$. Since $C_1*\{u_1,z_1\}$ is a 2-sphere, it follows that $\lk_\Delta u_2= C_1*\{u_1,z_1\}$ and $f_0(C_1)=a-2$. Hence $W_2=\{u_1,u_2\}$ and $W_1=V(C_1)\cup\{z_1\}\cup\{z_2\}$ for some vertex $z_2\in W_1$, so that $\lk_\Delta v=\{z_1,z_2\}*C_1$. Now assume that $\{z_1\} \notin \lk_\Delta u_2$ and $u_2$ is connected to vertices $u_3,u_4,\cdots,u_k$ in $\Delta[W_2]$. Since $C_1$ is a circle in the 2-sphere $\lk_\Delta u_2$, the subcomplex $\lk_\Delta u_2\backslash V(C_1)$ has two contractible connected components. If there is a vertex $u_i$ such that $\lk_{\lk_\Delta u_2} u_i=C_1$, then $\lk_\Delta u_2\supseteq C_1*\{u_1,u_i\}$ and hence this link is exactly $C_1*\{u_1,u_i\}$. This implies that $\deg_{\Delta[W_2]}u_2=2$. Otherwise, if $\lk_{\lk_\Delta u_2} u_i\neq C_1$ for all $3\leq i\leq k$, then each $u_i$ is connected to at least one vertex in $\lk_\Delta v\backslash (V(C_1)\cup\{z_1\})$. Since $\lk_\Delta u_1\supseteq \lk_\Delta v\backslash \{z_1\}$, it follows that the vertices $u_1$ and $u_3,\cdots,u_k$ are in the same connected component, and hence $\lk_\Delta u_2\backslash V(C_1)$ is connected, a contradiction.
			
			By applying the above argument inductively, we obtain that $\Delta[W_2\backslash\{v\}]$ is a path graph $(u_1,u_2,\cdots, u_{n-a-1})$, and there is a vertex $z_2$ in $W_1$ such that $\lk_\Delta u_1=\{z_2,u_2\}*C_1$ and $\lk_\Delta v=C_1*\{z_1,z_2\}$. Furthermore, $C_1\subseteq \lk_\Delta u_i$ for all $u_i\in W_2$. Then we let $C_2$ be the cycle graph $(v,z_2,u_1,u_2,\cdots,u_{n-a-1},z_1)$. It follows that $\Delta=C_1*C_2$. Since $a=|C_1|+2=\lfloor \frac{n+4}{2} \rfloor$ or $\lceil \frac{n+4}{4} \rceil$, $C_1$ and $C_2$ must be cycles of length $\lfloor \frac{n}{2} \rfloor$ or $\lceil \frac{n}{2} \rceil$. This implies $\Delta=J_2(n)$.
			
			By Lemma \ref{f number:3-manifold}, the $f$-vector of $\Delta$ is uniquely determined by $f_0(\Delta)=n$ and $f_1(\Delta)$. This yields the result.
		\end{proof}
		
		\section{Counting edges of flag 5-manifolds}
		
		For even-dimensional flag simplicial spheres, we have the following weaker form of the flag upper bound conjecture, see Conjecture 18 in \cite{AH}:
		\begin{conjecture}\label{conj: even dim}
			Fix $m\geq 1$. For every flag $2m$-sphere $\Delta$ on $n$ vertices, we have $f_1(M)\leq f_1(J_m^*(n))$.
		\end{conjecture}
		
		Using almost the same argument as in the proof of Theorem \ref{theorem: 3-manifold}, we establish the following proposition.
		\begin{proposition}\label{prop}
			Let $\Delta$ be a flag $(2m+1)$-manifold on $n$ vertices. If Conjecture \ref{conj: even dim} holds for all flag $2i$-spheres with $1\leq i\leq m$, then $f_1(\Delta)\leq f_1(J_{m+1}(n))$. Equality holds only when $\Delta=J_{m+1}(n)$.
		\end{proposition}
		\begin{proof}
			The proof is by induction on $m$. The case $m=1$ is confirmed by Theorem \ref{theorem: 3-manifold}. Now assume that $m\geq 2$ and $\Delta$ is a flag $(2m+1)$-manifold on $n$ vertices. If $n\equiv q \mod{m}$ ($0\leq q < m$), then a simple computation shows that
			\begin{equation}\label{f_1 of J_r(n)}
			f_1(J_{m}(n))=\frac{m-1}{2m}n^2+n+\frac{q(q-m)}{2m},\ f_1(J_m^*(n))=\frac{m-1}{2m}(n-2)^2+3(n-2)+\frac{q(q-m)}{2m}.
			\end{equation}
			
			As in the proof of Lemma \ref{lemma: flag normal 3-pseudomanifolds}, we let $v$ be a vertex of maximum degree in $V(\Delta)$,  $a=f_0(\lk_\Delta v)$, $W_1=V(\lk_\Delta v)$ and $W_2=V(\Delta)\backslash V(\lk_\Delta v)$. Following the same argument as in Lemma \ref{lemma: flag normal 3-pseudomanifolds}, we obtain the following analog of (\ref{formular}):
			\begin{equation*}
			\begin{split}
			f_1(\Delta) &\leq f_1(J_{m}^*(a))+|W_2|\cdot\max_{w\in W_2} f_0(\lk_\Delta w) -(f_0(\Delta[W_2\backslash\{v\}])-1)\\
			&\stackrel{(\Diamond)}{\leq} \frac{m-1}{2m}(a-2)^2+3(a-2)+\frac{q(q-m)}{2i}+(n-a)a-(n-a-2) \\
			&=-\frac{m+1}{2m}\big(a-(\frac{mn}{m+1}+2)\big)^2+\frac{m}{2m+2}n^2+n+\frac{q(q-m)}{2r},
			\end{split}		
			\end{equation*}
			where in ($\Diamond$) we used our assumption that Conjecture \ref{conj: even dim} holds for flag $2m$-spheres.
			By (\ref{f_1 of J_r(n)}), \[f_1(J_{m+1}(n))=\frac{m}{2m+2}n^2+n+\frac{q(q-m-1)}{2r}.\] Hence we conclude that $f_1(\Delta)\leq f_1(J_{m+1}(n))$. Moreover, equality holds when $a=\left\lceil{\frac{mn}{m+1}}\right \rceil+2$ or $a=\left\lfloor{\frac{mn}{m+1}}\right\rfloor+2$, and $\Delta[W_2\backslash\{w\}]$ is a tree. If $n=4m+4$ or $4m+5$, then by Lemma \ref{lemma: few vertices}, $\Delta$ is either $J_{m+1}(4m+4)$ or $J_{m+1}(4m+5)$. If $n>4m+5$, then the tree $\Delta[W_2\backslash\{w\}]$ has at least one edge. We proceed with the same argument as in Theorem \ref{theorem: 3-manifold} to show that $\lk_\Delta v$ must be the suspension of a $(2m-1)$-sphere on either $\left\lceil{\frac{mn}{m+1}}\right \rceil$ or $\left\lfloor{\frac{mn}{m+1}}\right\rfloor$ vertices. By induction, this sphere is the join of $m$ circles, each having length either $\left\lceil{\frac{n}{m+1}}\right \rceil$ or $\left\lfloor{\frac{n}{m+1}}\right\rfloor$. Hence $\Delta=J_{m+1}(n)$.
		\end{proof}
		\begin{theorem}\label{thm: 5-manifold}
			Let $\Delta$ be a flag 5-manifold on $n$ vertices. Then $f_1(\Delta)\leq f_1(J_3(n))$. Equality holds if and only if $\Delta=J_3(n)$.
		\end{theorem}
		\begin{proof}
			The result follows from Proposition \ref{prop} and the fact that Conjecture \ref{conj: even dim} is known to hold in the case of dimension four (see \cite{G}, Theorem 3.1.3).
		\end{proof}
		
		\section{The face numbers of flag 3-dimensional Eulerian complexes}
		In Lemma \ref{lemma: flag normal 3-pseudomanifolds}, we established an upper bound on the number of edges for all flag normal 3-pseudomanifolds. In this section, we find sharp upper bounds on the face numbers for all flag 3-dimensional Eulerian complexes. The proof relies on the following three lemmas.
		\begin{lemma}\label{dual edge}
		Let $\Delta$ be a flag $(d-1)$-dimensional simplicial complex.
				\begin{enumerate}[(a)]
					\item If $\sigma_1$ and $\sigma_2$ are two ridges that lie in the same facet $\sigma$ in $\Delta$, then the links of $\sigma_1$ and $\sigma_2$ are disjoint. 
					\item If $\sigma=\tau_1\cupdot\tau_2$ is a face in $\Delta$, then $V(\lk_\Delta \tau_1)\cap V(\lk_\Delta \tau_2)=V(\lk_\Delta \sigma)$. In particular, if $\sigma$ is a facet, then $f_0(\lk_\Delta \tau_1)+f_0(\lk_\Delta \tau_2)\leq f_0(\Delta)$. 
				\end{enumerate}
			\end{lemma}
			\begin{proof}
				For part (a), if $v$ is a common vertex of $\lk_\Delta \sigma_1$ and $\lk_\Delta \sigma_2$, then $v$ must be adjacent to each vertex of $\sigma_1\cup \sigma_2=\sigma$. Thus, since $\Delta$ is flag, $\{v\}\cup\sigma\in \Delta$, which contradicts our assumption that $\sigma$ is a facet.
				
				For part (b), the inclusion $V(\lk_\Delta \tau_1)\cap V(\lk_\Delta \tau_2)
				\supseteq V(\lk_\Delta \sigma)$ holds for any simplicial complex. If $v\in V(\lk_\Delta \tau_1)\cap V(\lk_\Delta \tau_2)$, then $v\cup\tau_1, v\cup\tau_2\in \Delta$. Since $\Delta$ is flag, it follows that $v\cup\sigma\in \Delta$. If $\sigma$ is not a facet, then $v\in \lk_\Delta \sigma$. This implies $V(\lk_\Delta \tau_1)\cap V(\lk_\Delta \tau_2)
				\subseteq V(\lk_\Delta \sigma)$. However, if $\sigma$ is a facet, then $v\cup\sigma$ cannot a facet in $\Delta$. In this case, $V(\lk_\Delta \tau_1)\cap V(\lk_\Delta \tau_2)=V(\lk_\Delta \sigma)=\emptyset$, and so $f_0(\lk_\Delta \tau_1)+f_0(\lk_\Delta \tau_2)\leq f_0(\Delta)$. 
			\end{proof}
			
			Lemma \ref{dual edge} part (b) implies that if $\Delta$ is a flag 3-dimensional simplicial complex and $\sigma\in \Delta$ is a facet, then $\sum_{e\subseteq \sigma} f_0(\lk_\Delta e)\leq 3f_0(\Delta)$, where the sum is over the edges of $\sigma$. The following lemma suggests a better estimate on $\sum_{e\subseteq \sigma} f_0(\lk_\Delta e)$ if $\Delta$ is a flag weak 3-pseudomanifold. 
			\begin{lemma}\label{sum of link of edges in a facet}
				Let $\Delta$ be a flag weak 3-pseudomanifold on $n$ vertices. Then for any facet $\sigma=\{v_1,v_2,v_3,v_4\}$ of $\Delta$, $\sum_{e\subseteq \sigma} f_0(\lk_\Delta e)\leq n+16$, where the sum is over the edges of $\sigma$. If equality holds, then $\cup_{w\in \tau}V(\lk_\Delta w)=V(\Delta)$ for any ridge $\tau\subseteq\sigma$.
			\end{lemma}
			\begin{proof}
				Let $V_i=V(\lk_\Delta v_i)$ for $1\leq i\leq 4$. By Lemma \ref{dual edge} part (b), for any distinct $1\leq i,j\leq 4$, we have $V_i\cap V_j=V(\lk_\Delta \{v_i,v_j\})$ and $V_1\cap V_2\cap V_3\cap V_4=V(\lk_\Delta \sigma)=\emptyset.$ Also since $\Delta$ is a weak 3-pseudomanifold, any ridge of $\Delta$ is contained in exactly two facets. Hence $V_i\cap V_j\cap V_k=V(\lk_\Delta \{v_i,v_j,v_k\})$ is a set of cardinality two. By the inclusiong-exclusion principle, we obtain that
				\begin{align}\label{IE1}
				\begin{split}
				\sum_{1\leq i<j\leq 4}|V_i\cap V_j|&=-|V_1\cup V_2\cup V_3\cup V_4|+\sum_{1\leq i\leq 4}|V_i|+\sum_{1\leq i<j<k\leq 4}|V_i\cap V_j\cap V_k|-|V_1 \cap V_2\cap V_3\cap V_4|\\
				&=\sum_{1\leq i\leq 4}|V_i|-|V_1\cup V_2\cup V_3\cup V_4|+\binom{4}{3}\cdot 2\\
				&=(|V_1|+|V_2|-|V_1\cup V_2|)+(|V_3|+|V_4|-|V_3\cup V_4|)+|(V_1\cup V_2)\cap (V_3\cup V_4)|+8\\
				&=|V_1\cap V_2|+|V_3\cap V_4|+|(V_1\cup V_2)\cap (V_3\cup V_4)|+8.
				\end{split}
				\end{align}
				For simplicity, we denote the set $(V_1\cup V_2)\cap (V_3\cup V_4)$ as $\bar{V}$. Notice that by Lemma \ref{dual edge} part (b), any vertex $v\in \Delta$ belongs to at most one of the sets $V_1\cap V_2$ and $V_3\cap V_4$. We split the vertices of $\Delta$ into the following three types.
				\begin{enumerate}
					\item If $v\in V_1\cap V_2$ and $v\notin V_3\cup V_4$, or if $v\in V_3\cap V_4$ and $v\notin V_1\cup V_2$, then $v\notin \bar{V}$. Each of these vertices contributes 1 to the right-hand side of (\ref{IE1}).
					\item If $v\in V_i\cap V_j\cap V_k$ for some triple $\{i,j,k\}\subseteq [4]$, then $v$ belongs to either $V_1\cap V_2$ or $V_3\cap V_4$, and $v\in \bar{V}$. By Lemma \ref{dual edge} part (a), every pair of ridges in $\sigma$ has disjoint links. Since $|V_i\cap V_j\cap V_k|=2$, the number of such vertices is exactly 8, and each of them contributes 2 to the right-hand side of (\ref{IE1}). 
					\item If $v\notin V_1\cap V_2$ and $v\notin V_3\cap V_4$, then $v$ contributes to the right-hand side of (\ref{IE1}) at most 1. This case occurs only when $v\in \bar{V}$, that is, when $v$ belongs to one of $V_1$ and $V_2$, and one of $V_3$ and $V_4$.
				\end{enumerate}
				Hence $\sum_{\{i,j\}\subseteq [4]}|V_i\cap V_j|\leq n+8+8=n+16$. Furthermore, if equality holds, then for every vertex $v$ in $\Delta$, either $v\in V_1\cap V_2$, or $v\in V_3\cap V_4$, or $v\in \bar{V}$. This implies that every vertex in $\Delta$ belongs to at least two of the four links $\lk_\Delta v_1,\cdots, \lk_\Delta v_4$. This proves the second claim.
			\end{proof}
			\begin{lemma}\label{sum of link of vertices in a facet}
				Let $\Delta$ be a flag weak 3-pseudomanifold on $n$ vertices, and let $\sigma=\{v_1,v_2,v_3,v_4\}$ be an arbitrary facet of $\Delta$. Then $\sum_{1\leq i\leq 4} f_0(\lk_\Delta v_i)\leq 2n+8$. If equality holds, then $\cup_{w\in \tau}V(\lk_\Delta w)=V(\Delta)$ for any ridge $\tau\subseteq\sigma$.
			\end{lemma}
			\begin{proof}
				As in Lemma \ref{sum of link of edges in a facet}, we let $V_i=V(\lk_\Delta v_i)$. By the inclusion-exclusion principle,
				\begin{multline*}
				\begin{split}
				\sum_{1\leq i\leq 4} |V_i| &= \sum_{1\leq i<j\leq 4}|V_i\cap V_j|- \sum_{1\leq i<j<k\leq 4}|V_i\cap V_j\cap V_k|+|V_1 \cap V_2\cap V_3\cap V_4|+|V_1\cup V_2\cup V_3\cup V_4|\\
				&= \sum_{1\leq i<j\leq 4}|V_i\cap V_j|+ |V_1\cup V_2\cup V_3\cup V_4| -8\\
				&\leq n+16+n-8=2n+8.
				\end{split}
				\end{multline*}
				The last inequality follows from Lemma \ref{sum of link of edges in a facet} and the fact that $|V_1\cup V_2\cup V_3\cup V_4|\leq |V(\Delta)|=n$. The second claim also follows from Lemma \ref{sum of link of edges in a facet}.
			\end{proof}
			
			Now we are ready to prove the main result in this section.
			\begin{theorem}\label{theorem: flag weak 3-pseudomanifolds}
				Let $\Delta$ be a flag 3-dimensional Eulerian complex on $n$ vertices. Then $f_1(\Delta)\leq f_1(J_2(n))$.
			\end{theorem}
			\begin{proof}
				We denote the vertices of $\Delta$ by $v_1,v_2,\cdots, v_n$, and we let $a_i=f_0(\lk_\Delta v_i)$. Since $\lk_\Delta v_i$ is a 2-dimensional Eulerian complex, the $f$-numbers of $\lk_\Delta v_i$ satisfy the relations
				\[f_2-f_1+f_0=2,\enspace 3f_2=2f_1.\]
				Hence $f_2(\lk_\Delta v_i)=2a_i-4$. By double counting, we obtain that
				\begin{equation}\label{double counting}
				\sum_{\sigma\in \Delta, |\sigma|=4}\sum_{v\in \sigma} f_0(\lk_\Delta v)=\sum_{i=1}^n f_0(\lk_\Delta v_i)\cdot\#\{\sigma:v_i\in \sigma, |\sigma|=4\}=\sum_{i=1}^{n}a_i(2a_i-4).
				\end{equation}
				By Lemma \ref{sum of link of vertices in a facet}, the left-hand side of (\ref{double counting}) is bounded above by $f_3(\Delta)(2n+8)$, which also equals $(f_1(\Delta)-n)(2n+8)$ by Lemma \ref{f number:3-manifold}. However, since $2f_1(\Delta)=\sum_{i=1}^n f_0(\lk_\Delta v_1)=\sum_{i=1}^{n}a_i$, the right-hand side of (\ref{double counting}) is bounded below by $n\cdot\frac{2f_1(\Delta)}{n}\cdot \big(\frac{4f_1(\Delta)}{n}-4 \big)$, and equality holds only if $a_i=\frac{2f_1(\Delta)}{n}$ for all $1\leq i\leq n$. Hence,
				\[(f_1(\Delta)-n)(2n+8)\geq n\cdot\frac{2f_1(\Delta)}{n}\cdot\Big(\frac{4f_1(\Delta)}{n}-4\Big).
				\]
				We simplify this inequality to get
				\[(f_1(\Delta)-n)\Big(\frac{8}{n}f_1(\Delta)-(2n+8)\Big)\leq 0.\]
				Since $f_1(\Delta)\geq n$, it follows that $f_1(\Delta)\leq\left\lfloor{\frac{n^2}{4}}\right\rfloor+n$, that is, $f_1(\Delta)\leq f_1(J_2(n))$. 
			\end{proof}
			
			The following corollary provides some properties of the maximizers of the face numbers in the class of flag 3-dimensional Eulerian complexes.
			\begin{corollary}\label{cor: equality case weak 3-pseudomanifold}
				Let $\Delta$ be a flag 3-dimensional Eulerian complex on $n$ vertices. Then $f_1(\Delta)=f_1(J_2(n))$ if and only if (i) $\left \lceil {\frac{n}{2}}\right\rceil$ vertices of $\Delta$ satisfy $f_0(\lk_\Delta v)=\left \lfloor {\frac{n}{2}}\right\rfloor+2$ while $\left \lfloor {\frac{n}{2}}\right\rfloor$ vertices satisfy $f_0(\lk_\Delta v)=\left \lceil {\frac{n}{2}}\right\rceil +2$, and (ii) $\Delta$ and all of its vertex links are connected.
			\end{corollary}
			\begin{proof}
				Part (i) of the claim follows from the proof of Theorem \ref{theorem: flag weak 3-pseudomanifolds}. Also by Theorem \ref{theorem: flag weak 3-pseudomanifolds}, if $f_1(\Delta)=f_1(J_2(n))$, then $\sum_{v\in \sigma}f_0(\lk_\Delta v)=2n+8$ for every facet $\sigma=\{v_1,v_2,v_3,v_4\}\in \Delta$. By Lemma \ref{sum of link of vertices in a facet}, every vertex of $\Delta$ belongs to $\cup_{1\leq i\leq 3}V(\lk_\Delta v_i)$, and hence $\Delta$ is connected. If there is a vertex $v$ such that $\lk_\Delta v$ is not connected, then we let $\tau_1=\{u_1,u_2,u_3\}$ and $\tau_2=\{w_1,w_2,w_3\}$ be two 2-faces in distinct connected components of $\lk_\Delta v$. Since $\Delta$ is flag,  by Lemma \ref{flag prop}, $\lk_\Delta v=\Delta[V(\lk_\Delta v)]$. Therefore no edges exist between $\tau_1$ and $\tau_2$. However, Theorem \ref{theorem: flag weak 3-pseudomanifolds} and Lemma \ref{sum of link of vertices in a facet} also imply that $\cup_{1\leq i\leq 3}V(\lk_\Delta u_i)=V(\Delta)$, contradicting the fact that $\{w_1,w_2,w_3\}\nsubseteq V(\lk_\Delta u_i)$ for all $1\leq i\leq 3$. Hence every vertex link in $\Delta$ is connected.
			\end{proof}
			
			The next lemma, which might be of interest in its own right, provides a sufficient condition for a flag complex to be the join of two of its links. 
			\begin{lemma}\label{equality case}
				Let $\Delta$ be a flag $(d-1)$-dimensional simplicial complex. If $\sigma=\tau_1\cupdot\tau_2$ is a facet of $\Delta$, where $\tau_1$ is an $i$-face of $\Delta$ and $\tau_2$ is a $(d-i-2)$-face of $\Delta$, then $V(\lk_\Delta \tau_1)\cup V(\lk_\Delta \tau_2)=V(\Delta)$ implies that $\Delta\subseteq \lk_\Delta \tau_1 * \lk_\Delta \tau_2$. Moreover, if $\Delta$ is a flag normal $(d-1)$-pseudomanifold, then $V(\lk_\Delta \tau_1)\cup V(\lk_\Delta \tau_2)=V(\Delta)$ if and only if $\Delta=\lk_\Delta \tau_1 * \lk_\Delta \tau_2$.
			\end{lemma}
			\begin{proof}
				Since $\Delta$ is flag, $\Delta[V(\lk_\Delta \tau_j)]=\lk_\Delta \tau_j$ for $j=1,2$. Hence for every $i$-face $\tau_2'$ in $\lk_\Delta \tau_1$, the link $\lk_\Delta \tau_2'$ does not contain any vertex in $\lk_\Delta \tau_1$. This implies $\lk_\Delta \tau_2'\subseteq \lk_\Delta \tau_2$. Similarly, for every $(d-i-2)$-face $\tau_1'\in \lk_\Delta \tau_2$, we have $\lk_\Delta \tau_1'\subseteq \lk_\Delta \tau_1$. Thus, $\Delta\subseteq \lk_\Delta \tau_1 * \lk_\Delta \tau_2$.
				
				If $\Delta$ is a normal pseudomanifold, then both $\lk_\Delta \tau_2$ and $\lk_\Delta \tau_2'$ are normal pseudomanifolds. Since no proper subcomplex of a normal pseudomanifold can be a normal pseudomanifold of the same dimension, it follows that $\lk_\Delta \tau_2'=\lk_\Delta \tau_2$. Similarly, for every $(d-i-2)$-face $\tau_1'\in \lk_\Delta \tau_2$, $\lk_\Delta \tau_1'=\lk_\Delta \tau_1$. Hence $\Delta=\lk_\Delta \tau_1*\lk_\Delta \tau_2$.
			\end{proof}
			\begin{remark}\label{remark: weak not normal}
				The second result in Lemma \ref{equality case} does not hold for flag weak pseudomanifolds, even assuming connectedness. Indeed, let $L_1,\cdots,L_4$ be four distinct circles of length $\geq 4$. Then $\Delta=(L_1*L_3)\cup(L_2*L_3)\cup (L_1*L_4)$ is a flag weak 3-pseudomanifold. If $\tau_1$ and $\tau_2$ are edges in $L_1$ and $L_3$ respectively, then $\lk_\Delta \tau_1=L_3\sqcup L_4$ and $\lk_\Delta \tau_2=L_1\sqcup L_2$. Hence $V(\lk_\Delta \tau_1)\cup V(\lk_\Delta \tau_2)=V(\Delta)$. However, $\Delta$ is a proper subcomplex of $\lk_\Delta \tau_1 *\lk_\Delta \tau_2$.
			\end{remark} 
			In Theorem \ref{theorem: 3-manifold} we proved that the maximizer of the face numbers is unique in the class of flag 3-manifolds on $n$ vertices. Is this also true for flag 3-dimensional Eulerian complexes? Corollary \ref{cor: equality case weak 3-pseudomanifold} implies that if the case of equality is not a join of two circles, then some of its edge links are not connected. Motivated by the example in Remark \ref{remark: weak not normal}, we construct a family of flag 3-dimensional Eulerian complexes on $n$ vertices that have the same $f$-numbers as those of $J_2(n)$.
			\begin{example}\label{Example: Eulerian maximizers}
				We write $C_i$ to denote a circle of length $i$. For a fixed number $n\geq 8$, let $a_1,a_2,\cdots, a_s$, $b_1,b_2,\cdots, b_t\geq 4$ be integers such that \[\sum_{1\leq i\leq s}a_i=\left\lfloor\frac{n}{2}\right\rfloor \enspace \mathrm{and}\; \sum_{1\leq j \leq t}b_j=\left\lceil\frac{n}{2}\right\rceil.\]
				We claim that $\Delta=\bigcup_{1\leq i\leq s, 1\leq j\leq t}(C_{a_i}*C_{b_j})$ is flag and Eulerian, where all $C_\cdot$ have disjoint vertex sets. Since the circles $C_{a_i}$ and $C_{b_j}$ are of length $\geq 4$, they are flag and hence $\Delta$ is also flag. Also any ridge $\tau$ in $\Delta$ can be expressed as $\tau=v\cup e$, where $v$ is a vertex of $C_{a_i}$ and $e$ is an edge of $C_{b_j}$ (or $v\in C_{b_j}$ and $e\in C_{a_i}$) for some $i,j$. By the construction of $\Delta$, the ridge $\tau$ is contained in exactly two facets $\{v,v'\}\cup e$ and $\{v,v''\}\cup e$ of $\Delta$, where $v'$ and $v''$ are neighbors of $v$ in the circle $C_{a_i}$ (or $C_{b_j}$). Hence the links of ridges in $\Delta$ are Eulerian. Since the edge links in $\Delta$ are either circles or disjoint unions of circles, and the vertex links in $\Delta$ are suspensions of disjoint union of circles, these links are also Eulerian. Finally, the vertices in $C_{a_i}$ have degree $\left\lceil\frac{n}{2}\right\rceil+2$ and the vertices in $C_{b_j}$ have degree $\left\lfloor\frac{n}{2}\right\rfloor$, and thus $f_1(\Delta)=f_1(J_2(n))$. A simple computation also shows that $f_2(\Delta)=f_2(J_2(n))$ and $f_3(\Delta)=f_3(J_2(n))$. Hence $\chi(\Delta)=\chi(J_2(n))$, and $\Delta$ is Eulerian.
			\end{example}
			
			We denote the set of all complexes on $n$ vertices constructed in Example \ref{Example: Eulerian maximizers} as $GJ(n)$. It turns out that $GJ(n)$ is exactly the set of maximizers of the face numbers in the class of flag 3-dimensional Eulerian complex on $n$ vertices. To prove this, we begin with the following lemma.
			\begin{lemma}\label{lem: suspension of disjoint circles}
				Let $\Delta$ be a flag 3-dimensional Eulerian complex on $n$ vertices. If $f_1(\Delta)=f_1(J_2(n))$, then every vertex link is the suspension of disjoint union of circles.
			\end{lemma}
			\begin{proof}
				Assume that $v$ is an arbitrary vertex of $\Delta$ and denote $\lk_\Delta v$ by $\Gamma$. Since $f_1(\Delta)=f_1(J_2(n))$, by the proof of Theorem \ref{theorem: flag weak 3-pseudomanifolds} and Lemma \ref{sum of link of vertices in a facet}, it follows that for every 2-face $\{v_1,v_2,v_3\}\in\Gamma$, we have $\cup_{1\leq i\leq 3}V(\lk_\Delta v_i)=V(\Delta)$. In particular, \[\cup_{1\leq i\leq 3}V(\lk_\Gamma v_i)=\cup_{1\leq i\leq 3}V(\lk_\Delta v_i[V(\Gamma)])=V(\Gamma).\]
				Since $f_0(\lk_\Gamma v_i\cap \lk_\Gamma v_j)=2$ for $1\leq i<j\leq 3$ and $f_0(\cap_{i=1}^{3}\lk_\Gamma v_i)=0$, by the inclusion-exclusion principle,
				\[\sum_{i=1}^{3}f_0(\lk_\Gamma v_i)=f_0(\Gamma)+\sum_{1\leq i< j\leq 3}f_0(\lk_\Gamma v_i\cap \lk_\Gamma v_j)-f_0(\cap_{i=1}^{3}\lk_\Gamma v_i)=f_0(\Gamma)+6.\]
				Also since $\Gamma$ is Eulerian, the $f$-vector of $\Gamma$ is $(f_0(\Gamma), 3f_0(\Gamma)-6, 2f_0(\Gamma)-4)$. Moreover, every vertex link in $\Gamma$ is the disjoint union of circles, and hence  $f_0(\lk_\Gamma v)=f_1(\lk_\Gamma v)$. By double counting,
				\begin{equation}\label{equ: double-counting}
					\sum_{\sigma\in\Gamma,|\sigma|=3}\sum_{v\in\sigma}f_0(\lk_\Gamma v)=\sum_{v\in\Gamma}f_0(\lk_\Gamma v)\cdot\#\{\sigma:v\in\sigma, |\sigma|=3\}=\sum_{v\in\sigma}f_0(\lk_\Gamma v) ^2.
				\end{equation}
				The left-hand side of (\ref{equ: double-counting}) equals $f_2(\Gamma)(f_0(\Gamma)+6)=2(f_0(\Gamma)-2)(f_0(\Gamma)+6)$. However, since $\sum_{v\in\sigma}f_0(\lk_\Gamma v)=2f_1(\Gamma)=6f_0(\Gamma)-12$, and every vertex link in $\Gamma$ has at least four vertices, the right-hand side of (\ref{equ: double-counting}) is bounded above by $2(f_0(\Gamma)-2)+(f_0(\Gamma)-2)\cdot 4^2=2(f_0(\Gamma)-2)(f_0(\Gamma)+6)$. Then the equality forces the right-hand side to obtain its maximum. Therefore, there exists two vertices $u_1,u_2\in\Gamma$ whose vertex links in $\Gamma$ have $f_0(\Gamma)-2$ vertices and the other vertex links have 4 vertices. If $f_0(\Gamma)=6$, then $\Gamma$ is the cross-polytope. Else if $f_0(\Gamma)>6$, then $f_0(\lk_\Gamma u_1)>4$. Since $\Gamma$ is flag, by Lemma \ref{flag prop}, $\Gamma[V(\lk_\Gamma u_1)]=\lk_\Gamma u_1$, and hence every vertex of $\lk_\Gamma u_1$ is not connected to $ f_0(\lk_\Gamma u_1)-3>1$ vertices in $\lk_\Gamma u_1$. This implies that $u_2\notin\lk_\Gamma u_1$. Hence $\lk_\Gamma u_2=\Gamma[V(\Gamma)\backslash\{u_1,u_2\}]=\lk_\Gamma u_1$, and $\Gamma$ is the join of $\lk_\Gamma u_1$ and two vertices $u_1,u_2$.
			\end{proof}
			\begin{theorem}\label{thm: Eulerian, case of equality}
				Let $\Delta$ be a flag 3-dimensional Eulerian complex on $n$ vertices. If $f_1(\Delta)=f_1(J_2(n))$, then $\Delta \in GJ(n)$.
			\end{theorem}
			\begin{proof}
				 By Lemma \ref{lem: suspension of disjoint circles}, we may assume that the link of vertex $v_1\in\Delta$ is the join of $C$ and two other vertices $v_2, v_3$, where $C$ is the disjoint union of circles. Then again by Lemma \ref{lem: suspension of disjoint circles}, the link of vertex $v_2$ is also the suspension of $C$. If $v_1'$ is any vertex of $C$ and its adjacent vertices in $C$ are $v_2', v_3'$, then by Lemma \ref{flag prop}, $\Delta[V(C)]=C$, and it follows that $f_0(\lk_\Delta v_i'\cap C)=2$ for $i=1,2$. Hence for $1\leq i,j\leq 2$,\[f_0(\lk_\Delta \{v_i,v_j'\})=f_0(\lk_\Delta v_i\cap\lk_\Delta v_j')\leq f_0(C\cap\lk_\Delta v_j')+2=4.\]
				 Furthermore, $V(\lk_\Delta \{v_1',v_2'\})$ is disjoint from $V(\lk_\Delta \{v_1,v_2\})$. So we obtain that 
				 \[\sum_{e\subseteq \{v_1',v_2',v_1,v_2\}}f_0(\lk_\Delta e)\leq n+4\cdot 4=n+16,\]
				 where the sum is over the edges of $\{v_1',v_2',v_1,v_2\}$.
				 Since $f_1(\Delta)=f_1(J_2(n))$, by the proof of Theorem \ref{theorem: flag weak 3-pseudomanifolds} and Lemma \ref{sum of link of edges in a facet}, it follows that this sum is exactly $n+16$. Hence $V(\lk_\Delta \{v_1,v_2\})\cupdot V(\lk_\Delta \{v_1',v_2'\})=V(\Delta)$. By Lemma \ref{equality case}, $\Delta\subseteq\lk_\Delta \{v_1,v_2\}*\lk_\Delta \{v_1',v_2'\}$. We count the number of edges in $\Delta$ to get
				 \[f_1(J_2(n))=f_1(\Delta)\leq f_1(\lk_\Delta \{v_1,v_2\}*\lk_\Delta \{v_1',v_2'\})=f_0(\lk_\Delta \{v_1,v_2\})\cdot f_0(\lk_\Delta \{v_1',v_2'\})+n\leq f_1(J_2(n)).\]
				Thus $f_1(\Delta)=f_1(\lk_\Delta \{v_1,v_2\}*\lk_\Delta \{v_1',v_2'\})$,  and the edge links $\lk_\Delta \{v_1,v_2\}, \lk_\Delta \{v_1',v_2'\}$ must be disjoint unions of circles on $\left\lceil\frac{n}{2}\right\rceil$ and $\left\lfloor\frac{n}{2}\right\rfloor$ vertices respectively. Since the flag complex $\Delta$ is determined by its graph, it follows that $\Delta=\lk_\Delta \{v_1,v_2\}*\lk_\Delta \{v_1',v_2'\}$, i.e., $\Delta\in GJ(n)$.
			\end{proof}
		\begin{remark}
			Theorem \ref{thm: Eulerian, case of equality} implies Theorem \ref{theorem: 3-manifold}. This is because every 3-manifold is Eulerian and the only complex in $GJ(n)$ that is also a 3-manifold is $J_2(n)$.
		\end{remark}
		\section{Concluding Remarks}
        We close this paper with a few remarks and open problems.
        
        As mentioned in the introduction, Klee \cite{K} verified that the Motzkin's UBC for polytopes holds for Eulerian complexes with sufficiently many vertices, and conjectured it holds for all Eulerian complexes. Can the upper bound conjecture for flag spheres also be extended to flag Eulerian complexes? Motivated by Theorem \ref{theorem: 3-manifold} and Theorem \ref{theorem: flag weak 3-pseudomanifolds}, we posit the following conjecture in the same spirit as Problem 17(i) from \cite{AH}:
        \begin{conjecture}\label{conjecture: general upper bound}
        	Let $\Delta$ be a flag $(2m-1)$-dimensional complex, where $m\geq 2$. Assume further that $\Delta$ is an Eulerian complex on $n$ vertices. Then $f_i(\Delta)\leq f_i(J_m(n))$ for all $i=1,\cdots, 2m-1$.
        \end{conjecture}
        Theorem \ref{theorem: flag weak 3-pseudomanifolds} gives an affirmative answer in the case of $m=2$ and $1\leq i\leq 3$. The next case is $i=1$ and $m=3$. In this case, Theorem \ref{thm: 5-manifold} verifies Conjecture \ref{conjecture: general upper bound} for flag 5-manifolds. Since Gal's proof that the real rootedness conjecture holds for homology spheres of dimension less than five relies on the Davis-Okun Theorem, it appears that whether Theorem \ref{thm: 5-manifold} can be extended to the generality of flag Eulerian complexes may depend on whether the Davis-Okun Theorem continues to hold for a larger class of complexes.
        
        The above results and conjectures discuss odd-dimensional flag complexes. What happens in the even-dimensional cases? To this end, we pose the following strengthening of Conjecture 18 from \cite{AH}.
        
        Let $J_m^*(n):=\Sp^0*C_1*\cdots*C_m$, where each $C_i$ is a circle of length either $\left\lceil{\frac{n-2}{m}}\right \rceil$ or $\left\lfloor{\frac{n-2}{m}}\right \rfloor$, and the total number of vertices of $J_m^*(n)$ is $n\geq 4m+2$. Now we let $\mathcal{S}_{n}$ denote the set of flag 2-spheres on $n$ vertices, and define \[\mathcal{J}_m^*(n):=\{S*C_2*\cdots*C_m\,|\,S \in \mathcal{S}_{V(C_1)+2}\}.\] It is not hard to see that every element in $\mathcal{J}_m^*(n)$ is a flag $2m$-sphere. 
        \begin{conjecture}
        	Let $\Delta$ be a flag homology $2m$-sphere on $n$ vertices. Then $f_i(\Delta)\leq f_i(J_m^*(n))$ for all $i=1,\cdots, 2m$. If equality holds for some $1\leq i\leq 2m$, then $\Delta \in \mathcal{J}_m^*(n)$.
        \end{conjecture}

        \section*{Acknowledgements}
        I would like to thank Isabella Novik, Steven Klee and Eran Nevo for helpful comments.

		\bibliographystyle{amsplain}
		
	\end{document}